\documentclass[10pt,a4paper]{amsart}
\usepackage[english]{babel}
\usepackage{amsmath,amsthm,mathrsfs}
\usepackage{amssymb}
\usepackage{aliascnt}
\usepackage{comment}
\usepackage{todonotes}
\usepackage{xspace}
\usepackage{nicefrac}

\DeclareMathOperator{\dom}{dom}

\DeclareMathOperator{\cof}{cof}

\DeclareMathOperator{\len}{len}
\DeclareMathOperator{\crit}{crit}

\DeclareMathOperator{\rank}{rank}

\DeclareMathOperator{\cf}{cf}

\DeclareMathOperator{\hull}{Hull}

\DeclareMathOperator{\Ult}{Ult}

\newcommand{\ZF}{{\rm ZF}\xspace}
\newcommand{\ZFC}{{\rm ZFC}\xspace}

\newtheorem{theorem}{Theorem}
\newaliascnt{example}{theorem}

\aliascntresetthe{example}
\newaliascnt{fact}{theorem}
\newtheorem{fact}[fact]{Fact}
\aliascntresetthe{fact}
\newaliascnt{corollary}{theorem}
\newtheorem{corollary}[corollary]{Corollary}
\aliascntresetthe{corollary}
\newaliascnt{lemma}{theorem}
\newtheorem{lemma}[lemma]{Lemma}
\aliascntresetthe{lemma}
\newaliascnt{claim}{theorem}
\newtheorem{claim}[lemma]{Claim}
\aliascntresetthe{claim}
\newtheorem{prop}[theorem]{Proposition}
\theoremstyle{definition}
\newaliascnt{definition}{theorem}

\aliascntresetthe{definition}
\newtheorem{question}{Question}

\newtheorem*{theorem*}{Theorem}
\newtheorem*{remark}{Remark}
\newtheorem*{example*}{Example}

\newtheorem*{definition*}{Definition}

\DeclareMathOperator{\ran}{ran}
\DeclareMathOperator{\ot}{ot}

\DeclareMathOperator{\ord}{Ord}

\DeclareMathOperator{\ns}{NS}

\newcommand{\p}{\mathcal{P}}
\newcommand{\la}{\langle}
\newcommand{\ra}{\rangle}
\title{Embeddings into outer models}
\author{Monroe Eskew}
\address{Universität Wien \\
Institut für Mathematik \\
Kurt Gödel Research Center \\
Kolingasse 14-16 \\
1090 Wien \\
AUSTRIA}
\email[M. Eskew]{monroe.eskew@univie.ac.at}
\email[S.-D. Friedman]{sdf@logic.univie.ac.at}
\author{Sy-David Friedman}
\thanks{The authors wish to acknowledge the support of the Austrian Science Fund (FWF) through Research Project P28420.}

\begin{document}
\maketitle

\begin{abstract}
We explore the possibilities for elementary embeddings $j : M \to N$, where $M$ and $N$ are models of \ZFC with the same ordinals, $M \subseteq N$, and $N$ has access to large pieces of $j$.  We construct commuting systems of such maps between countable transitive models that are isomorphic to various canonical linear and partial orders, including the real line $\mathbb R$.
\end{abstract}

\section{Introduction}

The notion of an elementary embedding between transitive models of set theory is central to the investigation of principles of high consistency strength.  A typical situation involves an elementary $j : M \to N$, where $M$ and $N$ are models of \ZFC sharing the same ordinals, and $j$ is not the identity map.  Such a map cannot be the identity on ordinals, 
and the least ordinal moved is called the \emph{critical point}.  Postulating more agreement between $M$ and $N$, and with the ``real'' universe $V$, usually results in a stronger hypothesis.  Kunen established an upper bound to this collection of ideas, showing that there is no nontrivial elementary $j : V \to V$.

The question is how exactly to formalize Kunen's result, which is not on its face equivalent to a first-order statement.  Indeed, it cannot be about an \emph{arbitrary} elementary embedding which may exist in some outer universe, as the relatively low-strength assumption of $0^\sharp$ gives a nontrivial elementary $j : L \to L$.  Furthermore, requiring $j$ to be definable from parameters in $V$ yields an impossibility using more elementary considerations and not requiring the Axiom of Choice \cite{MR1780073}.  One way to express the content of Kunen's Theorem is to require that $V$ satisfies the expansion of \ZFC to include Replacement and Comprehension for formulas that use $j$ as a predicate.  Another way is to localize its combinatorial content, which shows that there cannot be a cardinal $\lambda$ and a nontrivial elementary $j : V_{\lambda+2} \to V_{\lambda+2}$.

We begin with an elaboration on Kunen's theorem in a slightly different direction.  Using the main idea of Woodin's proof of Kunen's theorem (see \cite{MR1994835}, Theorem 23.12), we show that whenever one model of \ZFC is embedded into another with the same ordinals, there are some general constraints on what the models can know about the embedding and about each other.

\begin{theorem}
\label{kunengen}
Suppose $j : M \to N$ is a nontrivial elementary embedding between transitive models of $\ZFC$ with the same class of ordinals $\Omega$.  Then at least one of the following holds:
\begin{enumerate}
\item\label{repl} The \emph{critical sequence} $\la j^n(\crit(j)) : n \in \omega \ra$ is cofinal in $\Omega$.
\item\label{comp} For some $\alpha \in \Omega$, $j[\alpha] \notin N$.
\item\label{cof} For some $\alpha \in \Omega$, $\alpha$ is regular in $M$ and singular in $N$.
\end{enumerate}
\end{theorem}

\begin{proof}
We will suppose that all of the alternatives fail and derive a contradiction.  Let $\lambda \in \Omega$ be the supremum of the critical sequence.  Since $j[\lambda] \in N$, the critical sequence is a member of $N$, and thus $N \models \cf(\lambda) = \omega$.  Since $\cf(\lambda)^M$ is regular in $M$, it is also regular in $N$, and so $M \models \cf(\lambda) = \omega$ as well.  Thus $j(\lambda) = \lambda$.  Since $(\lambda^+)^M$ is regular in $N$, we must have that $(\lambda^+)^M$ is also a fixed point.

Let $\kappa = \crit(j)$, and choose in $M$ a pairwise-disjoint sequence of stationary subsets of $\lambda^+ \cap \cof(\omega)$, $\la S_\alpha : \alpha < \kappa \ra$.  Let $\la S'_\alpha : \alpha < j(\kappa) \ra = j(\la S_\alpha : \alpha < \kappa \ra)$.  Let $C = \{ \alpha < \lambda^+ : j[\alpha] \subseteq \alpha \}$.   $C$ is a member of $N$ and a club in $\lambda^+$, so let $\alpha \in C \cap S'_\kappa$.  Since $N \models \cf(\alpha) = \omega$, there is a sequence $\la \alpha_n : n < \omega \ra \in M$ cofinal in $\alpha$.  Since $\alpha$ is closed under $j$, $j(\alpha) = \sup_{n<\omega} j(\alpha_n) = \alpha$.  By elementarity, there is some $\beta < \kappa$ such that $\alpha \in S_\beta$.  But then $\alpha \in S'_\beta \cap S'_\kappa = \emptyset$, a contradiction.
\end{proof}

Let us list some key examples of embeddings between transitive models of \ZFC with the same ordinals, in which \emph{exactly one} of the above alternatives holds:
\begin{enumerate}
\item Axiom I3 asserts the existence of a nontrivial elementary embedding $j : V_\lambda \to V_\lambda$, where $\lambda$ is a limit ordinal.  (\ref{comp}) and (\ref{cof}) must both fail for such $j$, so (\ref{repl}) must hold.
\item Suppose $\mathcal U$ is a countably complete ultrafilter over some set.  If $j : V \to M \cong \Ult(V,\mathcal U)$ is the ultrapower embedding, then (\ref{repl}) fails by the Replacement axiom, and (\ref{cof}) fails since $M \subseteq V$.  Thus (\ref{comp}) holds.
\item Situations in which (\ref{repl}) and (\ref{comp}) both fail will be considered in Sections \ref{fixedpts} and \ref{amcat} of the present paper.
\end{enumerate}

We will say that an elementary embedding $j : M \to N$ is \emph{(target-)amenable} if $j[x] \in N$ for all $x \in M$, or in other words that alternative (\ref{comp}) fails.  For \ZFC models, this is equivalent to saying that $j \cap \alpha^2 \in N$ for all ordinals $\alpha \in N$.  It is easy to see that $M \subseteq N$ in such a situation.  It is an immediate consequence of Theorem \ref{kunengen} that if $M$ and $N$ are transitive models of \ZFC with the same ordinals $\Omega$, and $j : M \to N$ is an amenable elementary embedding such that its critical sequence is \emph{not} cofinal in $\Omega$, then $M$ and $N$ do not agree on cofinalities.  The proof actually shows that $M$ and $N$ cannot agree on both the class of cardinals and the class $\{ \alpha : \cf(\alpha) = \omega \}$.

When the domain and target of an elementary embedding are the same, alternative (\ref{cof}) cannot hold.  If $j : M \to M$ is definable from parameters in some larger universe $V$, then alternatives (\ref{repl}) and (\ref{comp}) show that the closure of $M$, as measured by $V$, must run out at some point.  In contrast to the I3 examples that achieve amenability at the cost of countable closure, we show it is possible for such $M$ to be as near to $V$ as desired, and find in this motif a characterization of supercompactness:

\begin{theorem}
\label{agreement}
Suppose $\kappa \leq \lambda$ are regular.  $\kappa$ is $\lambda$-supercompact if and only if there is a $\lambda$-closed transitive class $M$ and a nontrivial elementary $j : M \to M$ with critical point $\kappa$. %and $\lambda<j(\kappa)$.  %Both $M$ and $j$ are definable from parameters.
\end{theorem}

Finally, we consider amenable embeddings for which alternative (\ref{repl}) does not hold.  Because of the importance of regular fixed points in the proof of Theorem \ref{kunengen}, we first explore the possible behaviors of cardinal fixed points of amenable embeddings, showing that essentially anything can happen.  Then we explore the possible structural configurations of commuting systems of amenable embeddings.

Given an ordinal $\delta$, let $\mathcal E_\delta$ be the category whose objects are all transitive models of \ZFC of height $\delta$ and whose arrows are all elementary embeddings between these models.  Let $\mathcal A_\delta$ be the subcategory where we take only amenable embeddings as arrows.  (It is easy to see that amenable embeddings are closed under composition.) 

Partial orders are naturally represented as categories where between any two objects there is at most one arrow, which we take to point from the lesser object to the greater.  We would like to know what kinds of partial orders can appear in a reasonable way as subcategories of an $\mathcal A_\delta$.  Let us say that a subcategory $\mathcal D$ of a category $\mathcal C$ is \emph{honest} if whenever $x$ and $y$ are objects of $\mathcal D$ and there is an arrow $f : x \to y$ in $\mathcal C$, then there is one in $\mathcal D$ as well.

\begin{theorem}
If there is a transitive model of \ZFC plus sufficiently many large cardinals,
\footnote{For (3), we use a set of measurable cardinals of ordertype $\omega+1$.  The others use hypotheses weaker than one measurable cardinal.} 
then there is a countable ordinal $\delta$ such that $\mathcal A_\delta$ contains honest subcategories isomorphic to:
\begin{enumerate}
\item The real numbers.
\item The complete binary tree of any countable height.
\item The reverse-ordered complete binary tree of height $\omega$.
\item An Aronszajn tree.
\item Every countable pseudotree.
\end{enumerate}
\end{theorem}

A pseudotree is a partial order that is linear below any given element, which generalizes both linear orders and trees.  These have been considered by several authors, for example in 
\cite{
MR2195726,
MR1173142,
MR0498152}.
In order to show the last item, we develop the model theory of pseudotrees.  We show that there is a countable pseudotree that has the same kind of universal property that the rationals have with respect to linear orders: It is characterized up to isomorphism by some first-order axioms, and every other countable pseudotree appears as a substructure.  We prove that for suitable $\delta$, the category $\mathcal A_\delta$ contains a copy of this universal countable pseudotree.

We also rule out some kinds of subcategories.  For example, if $\delta$ is countable, then $\mathcal A_\delta$ cannot contain a copy of $\omega_1$ or a Suslin tree.  There are many natural questions about the possible structure of these categories, and we list some of them at the end.

\section{Self-embeddings of highly closed classes}
\label{self}

In this section, we will prove Theorem \ref{agreement}.  First let us show that if $\lambda$ is regular, $M$ is a $\lambda$-closed transitive class, and $j : M \to M$ is a nontrivial elementary embedding with critical point $\kappa$,
then $\kappa$ is $\lambda$-supercompact.  First note that we may assume $j(\kappa) > \lambda$:  The proof of Theorem \ref{kunengen} shows that the critical sequence eventually must overtake $\lambda$.  For if not, then $\lambda \geq \eta^+$, where $\eta = \sup_{n<\omega} j^n(\kappa)$, and we can derive a contradiction from the assumption that $j[\eta^+] \in M$.  So composing $j$ with itself finitely many times yields an embedding that sends $\kappa$ above $\lambda$.

Next we claim that $\lambda^{<\kappa} = \lambda$, using a well-known argument (see \cite{MR2160657}).  Let $\vec C = \la C_\alpha : \alpha < \lambda \ra$ be such that $C_\alpha$ is a club in $\alpha$ of ordertype $\cf(\alpha)$.  Since $j[\lambda] \in M$ and $j(\lambda)$ is regular in $M$, $\gamma := \sup j[\lambda] < j(\lambda)$.  Let $C^* = j(\vec C)(\gamma)$.  Let $D = j^{-1}[C^*]$.  Since $j[\lambda]$ is ${<}\kappa$-closed, $|D| = \lambda$.  If $x \in [D]^{<\kappa}$, then $j(x) = j[x]$, and by elementarity, $x \subseteq C_\alpha$ for some $\alpha < \lambda$ such that $\cf(\alpha)<\kappa$.  Since $\kappa$ is inaccessible, $|\p(C_\alpha)| < \kappa$ when $|C_\alpha|<\kappa$.  Thus $\lambda^{<\kappa} \leq \lambda \cdot \kappa = \lambda$.

Thus, all subsets of $\p_\kappa\lambda$ are in $M$.  From $j$ we may define a $\lambda$-supercompactness measure in the usual way: $\mathcal U = \{ X \subseteq \p_\kappa\lambda : j[\lambda] \in j(X) \}$.

For the other direction, we use an iterated ultrapower.  Let $\mathcal U$ be a normal, fine, $\kappa$-complete ultrafilter on $\p_\kappa\lambda$.  Let $V = M_0$ and for $n<\omega$, let $j_{n,n+1} : M_n \to M_{n+1} = \Ult(M_n,j_{0,n}(\mathcal U))$ be the ultrapower embedding, and let $j_{m,n+1} = j_{n,n+1} \circ j_{m,n}$ for $m < n$.  Let $M_\omega$ the direct limit, and for $n<\omega$, let $j_{n,\omega}$ be the direct limit embedding.  Note that each $M_n$ is $\lambda$-closed, but the limit $M_\omega$ is not even countably closed.  $j_{0,\omega}(\kappa) = \sup_{n<\omega} j_{0,n}(\kappa)$, yet this ordinal is inaccessible in $M_\omega$.

To construct the desired model $M$, we find a generic for a Prikry forcing over $M_\omega$, which will restore $\lambda$-closure when adjoined.  
The sequences of classes $\la M_n : n < \omega \ra$ and embeddings $\la j_{m,n} : m < n < \omega \ra$ are definable in $V$ from $\mathcal U$.  Applying $j_{\mathcal U}$ to the sequences yields $\la M_n : 1 \leq n < \omega \ra$ and $\la j_{m,n} : 1\leq m < n < \omega \ra$, which has the same direct limit, $M_\omega$.  For any formula $\varphi(v_0,\dots,v_n)$ and parameters $a_0,\dots,a_n \in M_\omega$, $\varphi^{M_\omega}(a_0,\dots,a_n) \Leftrightarrow  \varphi^{M_\omega}(j_{\mathcal U}(a_0),\dots,j_{\mathcal U}(a_n))$.  Hence, $j_{\mathcal U} \ M_\omega$ is an elementary embedding into $M_\omega$.

Let us define the Prikry forcing $\mathbb P_{\mathcal U}$, which is standard.  Conditions take the form $\la x_0,\dots,x_n,A \ra$, where:
\begin{enumerate}
\item Each $x_i \in \p_\kappa\lambda$, and $\kappa_i := x_i \cap \kappa$ is inaccessible.
\item $x_i \subseteq x_{i+1}$, and $|x_i| < \kappa_{i+1}$.
\item $A \in \mathcal U$.
\end{enumerate}
Suppose $p = \la x_0,\dots,x_n,A \ra$ and $q =  \la x'_0,\dots,x'_m,B \ra$.  We say $q \leq p$ when:
\begin{enumerate}
\item $m \geq n$, and for $i \leq n$, $x_i = x'_i$.
\item For $n < i \leq m$, $x'_i \in A$.
\item $B \subseteq A$.
\end{enumerate}

Proof of the following can be found in \cite{MR2768695}:
\begin{lemma}
$\mathbb P_{\mathcal U}$ adds no bounded subsets of $\kappa$, collapses $\lambda$ to $\kappa$, and is $\lambda^+$-c.c.
\end{lemma}

We now argue that there exists an $M_\omega$-generic filter $G \subseteq j_{0,\omega}(\mathbb P_{\mathcal U})$ in $V$.  Furthermore, $M_\omega[G]$ is $\lambda$-closed.  The idea for $\kappa = \lambda$ is due to Dehornoy \cite{MR514228}.

For $n < \omega$, let $z_n = j_{n,\omega}[j_{0,n}(\lambda)]$.  Since $\crit(j_{n+1,\omega}) = j_{0,n+1}(\kappa) > j_{0,n}(\lambda)$, $z_n = j_{n+1,\omega}(j_{n,n+1}[j_{0,n}(\lambda)])\in \p_{j_{0,\omega}(\kappa)} (j_{0,\omega}(\lambda))^{M_\omega}$.  We define $G$ as the collection of $\la x_0,\dots,x_n,A \ra \in j_{0,\omega}(\mathbb P_{\mathcal U})$ such that $\la x_0,\dots,x_n\ra$ is an initial segment of $\la z_i : i < \omega \ra$ and $\{ z_i : n < i < \omega \} \subseteq A$.

\begin{lemma}$G$ is generic over $M_\omega$.

\begin{proof}
We use an analogue of Rowbottom's Theorem \cite{MR0323572}:  If $F : [\p_\kappa\lambda]^{<\omega} \to 2$, then there is a set $A \in \mathcal U$ and a sequence $r \in {^\omega}2$ such that whenever $\la x_i,\dots,x_n \ra \subseteq A$ is such that $x_i \subseteq x_{i+1}$ and $|x_i| < x_{i+1} \cap \kappa \in \kappa$, then $F(\{x_0,\dots,x_{n-1}\}) = r(n)$.  Let $D$ be a dense open subset of $j_{0,\omega}(\mathbb P_{\mathcal U})$ in $M_\omega$.  For each $s$ such that $s^\frown \la j_{0,\omega}(\p_\kappa\lambda)\ra \in j_{0,\omega}(\mathbb P_{\mathcal U})$, let $F_s : [\p_\kappa\lambda]^{<\omega} \to 2$ be defined by $F_s(t) = 1$ if there is $B_{s,t}$ such that $s^\frown t^\frown \la B_{s,t} \ra \in D$, and $F_s(t) = 0$ otherwise.  If $F_s(t) = 0$, define $B_{s,t} = j_{0,\omega}(\p_\kappa\lambda)$.  For each $s$, let $B_s$ be the diagonal intersection $\Delta_t B_{s,t} \in j_{0,\omega}(\mathcal U)$.  For each $s$, let $A_s$ and $r_s$ be given by the analogue of Rowbottom's Theorem.  Let $A^*$ be the diagonal intersection $\Delta_s A_s \cap B_s$.

We claim that for each $s$, whether a condition $s^\frown t ^\frown \la B\ra \leq s^\frown \la A^* \ra$ is in $D$ depends only on the length of $t$.  Let $s$ be given and let $t,t'$ be of the same length $n$, such that $s^\frown t ^\frown \la B\ra$ and $s^\frown t' \!^\frown \la B' \ra$ are both $\leq s^\frown \la A^* \ra$.  Then $t,t' \subseteq A_s$ and $B,B' \subseteq B_s$, so $B \subseteq B_{s,t}$ and $B' \subseteq B_{s,t'}$.  Thus $F_s(t) = F_s(t') = r_s(n)$, so either both $s^\frown t ^\frown \la B\ra$ and $s^\frown t' \!^\frown \la B' \ra$ are in $D$, or they are both are not in $D$.

Now if $D$ is a dense open subset of $j_{0,\omega}(\mathbb P_{\mathcal U})$ in $M_\omega$, there is some $n <\omega$ and $\bar D \in M_n$ such that $j_{n,\omega}(\bar D) = D$.  Let $A^*$ be as above, and let $m \geq n$ be such that $A^* = j_{m,\omega}(\bar A^*)$ for some $\bar A^* \in j_{0,m}(\mathcal U)$.  Let $l <\omega$ be such that for all sequences $t$ of length $\geq l$ and all $B$ such that $\la z_0,\dots,z_{m-1} \ra ^\frown t ^\frown \la B \ra \leq \la z_0,\dots,z_{m-1} \ra ^\frown \la A^* \ra$, we have $\la z_0,\dots,z_{m-1} \ra ^\frown t ^\frown \la B \ra \in D$.  Since $z_k \in A^*$ for all $k \geq m$, $D \cap G \not= \emptyset$.
\end{proof}
\end{lemma}

\begin{lemma}The map $j_{\mathcal U} \restriction M_\omega$ can be extended to a map $j : M_\omega[G] \to M_\omega[G]$.
\end{lemma}

\begin{proof}
Note that for each $z_n$, 
$$j_{\mathcal U}(z_n) = j_{0,1}(j_{n,\omega}[j_{0,n}(\lambda)]) = j_{n+1,\omega}[j_{1,n+1}(j_{0,1}(\lambda))] = j_{n+1,\omega}[j_{0,n+1}(\lambda)] = z_{n+1}.$$
Let $G'$ be the generic filter generated by $\la z_n : 1 \leq n < \omega \ra$.  Then $j_{\mathcal U}[G] \subseteq G'$.  Thus by Silver's criterion, we may extend the map to $j : M_\omega[G] \to M_\omega[G']$ by putting $j(\tau^G) = j_{\mathcal U}(\tau)^{G'}$ for every $j_{0,\omega}(\mathbb P_{\mathcal U})$-name $\tau$.  But clearly, $M_\omega[G'] = M_\omega[G]$.
\end{proof}

\begin{lemma}$M_\omega[G]$ is $\lambda$-closed.
\end{lemma}

\begin{proof}
It suffices to show that $M_\omega[G]$ contains all $\lambda$-sequences of ordinals.
Suppose $\la \xi_\alpha : \alpha < \lambda \ra \subseteq \ord$.  For each $\alpha$, there is $n < \omega$ and a function $f_\alpha : (\p_\kappa\lambda)^n \to \ord$ such that $\xi_\alpha = j_{0,\omega}(f_\alpha)(z_0,\dots,z_{n-1})$.  The sequence $\la j_{0,\omega}(f_\alpha) : \alpha < \lambda \ra$ can be computed from $j_{0,\omega}(\la f_\alpha : \alpha < \lambda \ra)$ and $j_{0,\omega}[\lambda]$, both of which are in $M_\omega$.  The sequence $\la \xi_\alpha : \alpha < \lambda \ra$ can be computed from $\la j_{0,\omega}(f_\alpha) : \alpha < \lambda \ra$ and $\la z_n : n < \omega \ra$, and is thus in $M_\omega[G]$.
\end{proof}

This concludes the proof of Theorem \ref{agreement}.

\section{Fixed point behavior of amenable embeddings}
\label{fixedpts}

One way to produce elementary embeddings is with indiscernibles, like the embeddings derived from $0^\sharp$.  If we want the embedding to be amenable to the target model, indiscernibles can also be used, more consistency strength is required.  Vickers and Welch \cite{MR1856729} showed a near-equiconsistency between the existence of an elementary $j : M \to V$, where $M$ is a transitive class and $V$ satisfies ZFC for formulas involving $j$, and the existence of a Ramsey cardinal.  In this section, we begin with the argument for constructing nontrivial amenable embeddings from a Ramsey cardinal, and then we elaborate on this idea using measurable cardinals to control the behavior of the embedding more precisely.

Recall that a cardinal $\kappa$ is \emph{Ramsey} when for every coloring of its finite subsets in $<\kappa$ colors, $c : [\kappa]^{<\omega} \to \delta < \kappa$, there is $X \in [\kappa]^\kappa$ such that $c \restriction [X]^n$ is constant for all $n<\omega$ ($X$ is \emph{homogeneous}).  Rowbottom \cite{MR0323572} showed that if $\kappa$ is measurable and $\mathcal U$ is a normal measure on $\kappa$, then for any coloring $c : [\kappa]^{<\omega} \to \delta < \kappa$, there is a homogeneous $X \in \mathcal U$.

Suppose $\kappa$ is Ramsey.  Let $\frak A$ be a structure on $V_\kappa$ in a language of size $\delta<\kappa$, that includes a well-order of $V_\kappa$ so that the structure has definable Skolem functions.  For $X \subseteq \frak A$, we write $\hull^{\frak A}(X)$ for the set $\{ f(z) : z \in X^{<\omega}$ and $f$ is a definable Skolem function for $\frak A \}$.  For $\alpha_0<\dots<\alpha_n<\kappa$, let $c(\alpha_0,\dots,\alpha_n) = \{ \varphi(v_0,\dots,v_n) : \frak A \models \varphi(\alpha_0,\dots,\alpha_n) \}$.  The number of colors is at most $2^\delta$, so let $X \in [\kappa]^\kappa$ be homogeneous for $c$.  If  $Y \subseteq X$ and $\xi \in X \setminus Y$, then $\xi \notin \hull^{\frak A}(Y)$:  For if not, let $f(v_0,\dots,v_n)$ be a definable Skolem function, and let $\{\alpha_0,\dots,\alpha_n\} \subseteq Y$ be such that $\xi = f(\alpha_0,\dots,\alpha_n)$.  Let $m\leq n$ be the maximum such that $\alpha_m < \xi$.  Suppose first that $m<n$.  Let $\alpha_{n+1} > \alpha_n$ be in $X$.  By homogeneity, 
\begin{align*}
\frak A \models & \xi = f(\alpha_0,\dots,\alpha_m,\alpha_{m+2},\dots,\alpha_n,\alpha_{n+1}), \text{ and } \\
\frak A \models & \alpha_{m+1} = f(\alpha_0,\dots,\alpha_m,\alpha_{m+2},\dots,\alpha_n,\alpha_{n+1}).
\end{align*}
This contradicts that $f$ is a function.  If $\xi > \alpha_n$, then similarly,  $A \models \xi = \xi' = f(\alpha_0,\dots,\alpha_n)$,
for some $\xi'>\xi$ in $X$, again a contradiction.
Furthermore, for every infinite $\mu \in [\delta,\kappa)$, $|\p(\mu) \cap \hull^{\frak A}(X)| \leq \mu$ \cite{MR0323572}.

Therefore, if $Y \subseteq X$ has size $\kappa$, and $M$ is the transitive collapse of $\hull^{\frak A}(Y)$, then $M$ is a proper transitive subset of $V_\kappa$ of size $\kappa$, and there is an elementary $j : M \to V_\kappa$.  $j$ is amenable simply because $V_\kappa$ has all sets of rank $<\kappa$, so in particular $j[x] \in V_\kappa$ for all $x \in M$.  Furthermore, if $Y_0,Y_1 \in [X]^\kappa$, then the order-preserving bijection $f : Y_0 \to Y_1$ induces an isomorphism $\pi : \hull^{\frak A}(Y_0) \to \hull^{\frak A}(Y_1)$, and thus these two hulls have the same transitive collapse $M$.  So given the structure $\frak A$, this process produces one proper transitive subset $M \subseteq V_\kappa$ of size $\kappa$, which can be amenably embedded into $V_\kappa$ in many different ways.  Let us examine the ways in which these embeddings may differ with regard to fixed points.

\begin{theorem}
\label{regfix}
Suppose $\kappa$ is measurable.  There is a transitive $M \subseteq V_\kappa$ of size $\kappa$ such that for every $\delta \leq \kappa$, there is an elementary embedding $j : M \to V_\kappa$ such that the set of $M$-cardinals fixed by $j$ above $\crit(j)$ has ordertype $\delta$.
\end{theorem}

\begin{proof}
Fix a normal ultrafilter $\mathcal U$ on $\kappa$ and some cardinal $\theta > 2^\kappa$.  Let $\frak A$ be a structure in a countable language expanding $(H_\theta,\in,\mathcal U)$ with definable Skolem functions.  Let $\frak A_0 \prec \frak A$ be such that $|\frak A_0| = |\frak A_0 \cap \kappa| < \kappa$.  We show something a little stronger than the claimed result; namely, for every $\delta \leq \kappa$, there is a set of indiscernibles $B \subseteq \kappa$ of size $\kappa$ such that if $\frak B = \hull(\frak A_0 \cup B)$, then $\frak B \cap \sup (\frak A_0 \cap \kappa) = \frak A_0 \cap \kappa$, and the set of cardinals in the interval $[\sup(\frak A_0 \cap \kappa),\kappa)$ that are fixed by the transitive collapse of $\frak B$ has ordertype $\delta$.  Thus if the inverse of the transitive collapse map of $\frak A_0$ has no fixed points above its critical point, for example if $\frak A_0$ is countable, then the result follows.

Let $\alpha_0 \in\bigcap (\mathcal U \cap \frak A_0)$ be strictly greater than $\sup(\frak A_0 \cap \kappa)$, and let $\frak A_1 = \hull(\frak A_0 \cup \{ \alpha_0 \})$.  We claim that $\frak A_0 \cap \kappa = \frak A_1 \cap \alpha_0$.  If $\gamma \in \frak A_1 \cap \kappa$, then there is a function $f : \kappa \to \kappa$ in $\frak A_0$ such that $f(\alpha_0) = \gamma$.  If $\gamma < \alpha_0$, then $f$ is regressive on a set in $\mathcal U$, and therefore constant on a set in $\mathcal U$, and thus $\gamma \in \frak A_0$.
Continue in this way, producing a continuous increasing sequence of elementary substructures of $\frak A$, $\la \frak A_i : i < \kappa \ra$, and an increasing sequence of ordinals $\la \alpha_i : i < \kappa \ra$, such that for $0<i<\kappa$, $\alpha_i = \min\bigcap(\mathcal U \cap \frak A_i)$, and $\frak A_{i+1} = \hull (\frak A_i \cup \{ \alpha_i \})$.  

\begin{claim}
$\{ \alpha_i : i < \kappa \} \in \mathcal U $.  
\end{claim}
\begin{proof}
Let $\frak A_\kappa = \bigcup_{i<\kappa} \frak A_i$, and let $\la X_i : i < \kappa \ra$ enumerate $\mathcal U \cap \frak A_\kappa$.  There is a club $C\subseteq\kappa$ such that for all $\beta \in C$, $\mathcal U \cap \frak A_\beta = \{ X_i : i < \beta \}$, and $\beta = \sup(\frak A_\beta \cap \kappa)$.  If $\beta \in \bigcap_{i<\beta} X_i$, then $\alpha_\beta = \beta$.  This means that $C \cap \Delta_{i<\kappa} X_i \subseteq \{ \alpha_i : i < \kappa \}$.  Since $\mathcal U$ is normal, the claim follows.
\end{proof}

Let $A = \{ \alpha_i : i < \kappa \}$.  Under a mild assumption on $\frak A_0$, $A$ is a set of order-indiscernibles for $\frak A$.  For let $\varphi$ be a formula in $n$ free variables in the language of $\frak A$.  Let $c_\varphi : [\kappa]^n \to 2$ be the coloring defined by $c_\varphi(a_1,\dots,a_n) = 1$ if $\frak A \models \varphi(a_1,\dots,a_n)$, and 0 otherwise.  Since $c_\varphi \in \frak A$, Rowbottom's Theorem implies that there is a set $X_\varphi \in \mathcal U$ such that $\frak A \models \varphi(a_1,\dots,a_n) \leftrightarrow  \varphi(b_1,\dots,b_n)$ whenever $\la a_1,\dots,a_n \ra,\la b_1,\dots,b_n \ra$ are increasing sequences from $X_\varphi$.  By slightly enlarging $\frak A_0$ at the beginning if necessary, we may assume $c_\varphi \in \frak A_0$ for each $\varphi$.  By elementarity, we may assume $X_\varphi \in \frak A_0$.  Thus $A \subseteq X_\varphi$ for each such $\varphi$.

Let $\la \beta_i : i < \kappa \ra$ be the increasing enumeration of the closure of $A \cup \{ \sup(\frak A_0 \cap \kappa) \}$.

\begin{claim}
\label{hullcontrol}
For every $X \subseteq A$ and $\gamma<\kappa$, if $\beta_\gamma \notin X$, then $\hull(\frak A_0 \cup X)$ is disjoint from the interval  $[\beta_\gamma,\beta_{\gamma+1})$.
\end{claim}
\begin{proof}
Suppose $\beta_\gamma \notin X$, $\xi \in [\beta_\gamma,\beta_{\gamma+1})$, but there are elements $c_1,\dots,c_k \in \frak A_0$, ordinals $i_1 < \dots < i_n$, and a Skolem function $f$ such that $\xi = f(c_1,\dots,c_k,\alpha_{i_1},\dots,\alpha_{i_n})$, with $\{ \alpha_{i_1},\dots,\alpha_{i_n} \} \subseteq X$.  Let us assume that we have chosen $f$ and $c_1,\dots,c_k$ to output $\xi$ with the least number $n$ of parameters from $X$.  
Since $\alpha_{i_n+1} > \sup(\frak A_{i_n+1} \cap \kappa)$, we have $\xi < \alpha_{i_n}$.  Working in $N = \hull(\frak A_0 \cup \{ \alpha_{i_1},\dots,\alpha_{i_{n-1}} \})$, let 
$$Y = \{ \eta < \kappa : f(c_1,\dots,c_k, \alpha_{i_1},\dots,\alpha_{i_{n-1}},\eta ) < \eta \}.$$
Since $\alpha_{i_n} \in Y \in \mathcal U$ and $\mathcal U$ is normal, there is $\zeta < \kappa$ and $Z \in \mathcal U \cap N$ such that $f(c_1,\dots,c_k, \alpha_{i_1},\dots,\alpha_{i_{n-1}},\eta ) = \zeta$ for all $\eta \in Z$.  Thus, $\zeta = \xi \in N$, and $\xi$ is the output of a Skolem function with only $n-1$ inputs from $X$, contrary to the minimality assumption.
\end{proof}

Now let $D \subseteq A \cap \lim A$ have ordertype $\delta$.  Let $B = (A \setminus \lim A) \cup D$.  Let $\frak B = \hull(\frak A_0 \cup B)$, let $H$ be the transitive collapse of $\frak B$,  let $M = H \cap V_\kappa$, and let $j : M \to V_\kappa$ be the inverse of the collapse map.

Let $C = j^{-1}[B]$, and let $\la B(i) : i < \kappa \ra$ and $\la C(i) : i < \kappa \ra$ denote the respective increasing enumerations of these sets.  Let $\lambda \in \lim C$.  By Claim \ref{hullcontrol}, there is some $i<\kappa$ such that $j(\lambda) = \beta_i$, and $\beta_i \in B$.  Thus $C$ is closed.
Since $|C(i)| = |C(i+1)|$ for all $i$, every cardinal of $V_\kappa$ above $C(0)$ is a limit point of $C$.

Let $\kappa_0  = \ot(\frak A_0 \cap \kappa)$.  Since $\frak B \cap \alpha_0 = \frak A_0 \cap \kappa$, $j(\kappa_0) = \alpha_0 > \kappa_0$.
Define an increasing continuous sequence $\la \kappa_i : i < \kappa \ra$
as follows.  Given $\kappa_i$, if $j(\kappa_i) > \kappa_i$, let $\kappa_{i+1} = j(\kappa_i)$, and if $j(\kappa_i) = \kappa_i$, let $\kappa_{i+1} = (\kappa_i^+)^M$.  For limit $\lambda$, let $\kappa_\lambda = \sup_{i<\lambda}\kappa_i$.

We claim that $D$ is the set of points above $\kappa_0$ that are fixed by $j$.  To show each point in $D$ is fixed, note that if $\xi \in D$, then $\xi$ is closed under $j$, because $\ot(B \cap \xi) = \xi$.  Thus $\xi = \kappa_\xi$, and $j(\xi) = \xi$.  To show no other points are fixed, we argue by induction that if $\xi \notin D$, then $j(\kappa_\xi) = \kappa_{\xi+1}$.  Assume that this holds for all $i<\xi$, and $\xi \notin D$.
\begin{itemize}
\item \underline{Case 1:} $\xi$ is a limit.  Then $\kappa_\xi = \sup_{i<\xi} \kappa_\xi$.  By the inductive assumption, $\kappa_i \in \ran(j)$ for unboundedly many $i <\xi$, so $\kappa_\xi$ is a cardinal in $V$, and thus $\kappa_\xi = C(\kappa_\xi)$.  Therefore, $j(\kappa_\xi) = B(\kappa_\xi)$.  If $\xi < \kappa_\xi$, then $\kappa_\xi \notin A$ since it is singular, so $B(\kappa_\xi) > \kappa_\xi$.  If $\xi = \kappa_\xi$, then $B(\xi) > \xi$ since $\xi \notin D$.  In either case, the definition of the sequence gives that $\kappa_{\xi+1} = j(\kappa_\xi)$.
\item \underline{Case 2:} $\xi = \eta+1$, and $\eta \in D$.  Then $\kappa_\xi = (\kappa_\eta^+)^M$.  Since $|\frak A_{\eta+1}| = \eta$, $(\kappa_\eta^+)^M < (\kappa_\eta^+)^V$.  Thus $j(\kappa_\xi) = (\kappa_\eta^+)^V = \kappa_{\xi+1}$.
\item \underline{Case 3:} $\xi = \eta+1$, and $\eta \notin D$.  Then by induction, $\kappa_\eta < j(\kappa_\eta) = \kappa_\xi$.  By elementarity, $\kappa_\xi = j(\kappa_\eta) < j(\kappa_\xi)$, so by definition, $\kappa_{\xi+1} = j(\kappa_\xi)$.
\end{itemize}
To conclude, if $\xi \in D$, then the interval $[\kappa_\xi,\kappa_{\xi+1})$ contains only one $M$-cardinal, and it is fixed.  If $\xi \notin D$, then the interval $[\kappa_\xi,\kappa_{\xi+1})$ is moved into the interval $[\kappa_{\xi+1},\kappa_{\xi+2})$.
\end{proof}

\begin{theorem}
\label{singfix}
Suppose $\kappa$ is a measurable limit of measurables, $\delta \leq \kappa$, and $f : \delta \to 2$.  Then there is a transitive $M \subseteq V_\kappa$ of size $\kappa$ and an elementary embedding $j : M \to V_\kappa$ such that the set of $M$-cardinals fixed by $j$ above $\crit(j)$ has ordertype $\delta$, and for all $\alpha < \delta$, the $\alpha^{th}$ fixed point is regular iff $f(\alpha) = 1$.
\end{theorem}

\begin{proof}
Let $\mathcal U$ be a normal measure on $\kappa$, let $\theta > 2^\kappa$, and let $\frak A$ be a structure in a countable language expanding  $(H_\theta,\in,\mathcal U)$ with definable Skolem functions.  Let us first show how to get isolated singular fixed points.  Suppose $\frak A_0 \prec \frak A$ and there is a cardinal $\mu\in \frak A_0 \cap \kappa$ such that $|\frak A_0|<\mu$.  Let $\la \delta_n : n < \omega \ra$ enumerate the first $\omega$ measurable cardinals in $\frak A_0 \cap \kappa$ above $\mu$.  By the previous construction, we may make a series of elementary extensions $\frak A_0 \prec \frak A_1 \prec \frak A_2 \prec \dots$ such that:
\begin{enumerate}
\item For all $n$, $\frak A_n \cap \delta_n = \frak A_{n+1} \cap \delta_n$.
\item $\ot(\frak A_1 \cap \delta_0) = | \frak A_1 | = \mu$.
\item For $n \geq 1$, $\ot(\frak A_n \cap \delta_n) = |\frak A_n | = \delta_{n-1}$.
\end{enumerate}
Then let $\frak A_\omega = \bigcup_{n<\omega} \frak A_n$.  Let $j : H \to \frak A_\omega$ be the inverse of the transitive collapse.  It is easy to see that $j(\mu) = \delta_0$, and for all $n<\omega$, $j(\delta_n) = \delta_{n+1}$.  Let $\delta_\omega = \sup_n \delta_n$, which is in $\frak A_0$.  Since $\ot(\frak A_\omega \cap \delta_\omega) = \delta_\omega$, $j(\delta_\omega) = \delta_\omega$.  Since $|\frak A_\omega| = \delta_\omega$, all cardinals of $H$ above $\delta_\omega$ are below $(\delta_\omega^+)^V$, so $j$ has no cardinal fixed points greater than $\delta_\omega$.

To prove the theorem, we build a continuous chain of elementary submodels of $\frak A$, $\la \frak A_\alpha : \alpha < \kappa \ra$, along with an increasing sequence of cardinals $\la \delta_\alpha : \alpha < \delta \ra \subseteq \kappa$, such that:
\begin{enumerate}
\item If $\alpha < \beta$, then $\frak A_\beta \cap \sup(\frak A_\alpha \cap \kappa)= \frak A_\alpha \cap \kappa$.
\item For each $\alpha$, $\delta_\alpha \in \frak A_{\alpha+1} \setminus \sup(\frak A_\alpha \cap \kappa)$.
\item If $\alpha < \delta$, then  $\delta_\alpha$ is the unique cardinal fixed point of the transitive collapse of $\frak A_{\alpha+1}$ that is $\geq\sup(\frak A_{\alpha} \cap \kappa)$, and $\delta_\alpha$ is regular iff $f(\alpha) = 1$.
\item If $\alpha \geq \delta$, then there are no fixed points in the transitive collapse of $\frak A_{\alpha+1}$ that are $\geq \sup(\frak A_{\alpha} \cap \kappa)$.
\end{enumerate}
Given $\frak A_\alpha$, first adjoin an ordinal $\gamma_\alpha \in \bigcap(\frak A_\alpha \cap \mathcal U)$ that is strictly above $\sup(\frak A_\alpha \cap \kappa)$, to form a model $\frak A_\alpha'$ with the same ordinals below $\gamma_\alpha$, so that if $\pi$ is the transitive collapse of $\frak A_\alpha'$, then $\pi(\gamma_\alpha)<\gamma_\alpha$.  If $\alpha<\delta$ and $f(\alpha) = 0$, let $\delta_\alpha$ be the supremum of the first $\omega$ measurable cardinals in $\frak A_\alpha'$ above $\gamma_\alpha$.   If $\alpha < \delta$ and $f(\alpha) = 1$, let $\delta_\alpha$ be the least measurable cardinal in $\frak A_\alpha'$ above $\gamma_\alpha$.  Depending on $f(\alpha)$, use either the above construction or that of Theorem \ref{regfix} to build an extension $\frak A_{\alpha+1}$ of $\frak A_\alpha'$ with size $\delta_\alpha$, so that $\delta_\alpha$ is the unique cardinal fixed point of its transitive collapse above $\sup(\frak A_\alpha \cap \kappa)$, and $\frak A_{\alpha+1} \cap \gamma_\alpha = \frak A_\alpha \cap \kappa$.

If $\delta = \kappa$, we continue the same process for $\kappa$-many steps. If $\delta < \kappa$, then once we have $\frak A_\delta$, we can apply Theorem \ref{regfix} to extend to a $\kappa$-sized model $\frak A_\kappa$, such that $\frak A_\kappa \cap \sup(\frak A_\delta \cap \kappa) = \frak A_\delta \cap \kappa$, and the transitive collapse of $\frak A_\kappa$ has no fixed points in the interval $[\sup(\frak A_\delta \cap \kappa),\kappa)$.
\end{proof}

\section{Categories of amenable embeddings}
\label{amcat}

In this section, we explore the structure of the categories $\mathcal A_\delta$ of models of \ZFC of height $\delta$ and amenable embeddings between them.  We attempt to determine, as comprehensively as we can, what kinds of linear and partial orders can appear as honest subcategories of $\mathcal A_\delta$, for a countable ordinal $\delta$.  We start by applying the methods of the previous section to build patterns of ``amalgamation'' and then use stationary tower forcing to build transfinite chains and patterns of ``splitting.''  Then we pass to ill-founded models that are well-founded up to $\delta$ to find subcategories of $\mathcal A_\delta$ that are isomorphic to certain canonical ill-founded structures.

The following large cardinal notion will be important for this section:
\begin{definition*}[\cite{MR2069032}]A cardinal $\kappa$ is called \emph{completely J\'onsson} if it is inaccessible and for every $a \in V_\kappa$ which is stationary in $\bigcup a$, the set 
$\{ M \subseteq V_\kappa : M \cap \bigcup a \in a \text{ and } |M \cap \kappa| = \kappa \}$
is stationary in $\p(V_\kappa)$.
\end{definition*}

The construction of the previous section shows that measurable cardinals are completely J\'onsson.  For let $\kappa$ be measurable and let $a \in V_\kappa$ be stationary.  For any function $F : V_\kappa^{<\omega} \to V_\kappa$, we can take a well-order $\lhd$ on $H_\theta$, where $\theta > 2^\kappa$, and find an elementary substructure $\frak A \prec (H_\theta,\in,\lhd,F)$ of size $<\kappa$, and such that $\frak A \cap \bigcup a \in a$.  By repeated end-extension, we get $\frak B \prec (H_\theta,\in,\lhd,F)$ such that $|\frak B \cap \kappa| = \kappa$ and $\frak B \cap \bigcup a = \frak A \cap \bigcup a$.  Since $\frak B \cap V_\kappa$ is closed under $F$, the claim follows.  An easy reflection argument shows that if $\kappa$ is measurable and $\mathcal U$ is a normal measure on $\kappa$, then the set of completely J\'onsson cardinals below $\kappa$ is in $\mathcal U$.

Similar arguments show the following: If $\la \kappa_i : i < \delta \ra$ is an increasing sequence of measurable cardinals with supremum $\theta$, and $\la \mu_i : i < \delta \ra$ is a nondecreasing sequence of cardinals with $\mu_i \leq \kappa_i$ for all $i$, then the set $\{ X \subseteq \theta : |X \cap \kappa_i| = \mu_i$ for all $i \}$ is stationary.

\subsection{Collapsing inward}

Following \cite{MR1940513}, let us say that a \emph{sequential tree} is a collection of functions $p : n \to \omega$ for $n<\omega$ closed under initial segments, ordered by $p \leq q$ iff $p \supseteq q$.  If $T$ is a well-founded sequential tree, then it has a rank function $\rho_T : T \to \omega_1$, defined inductively by putting $\rho_T(p) = 0$ for minimal nodes $p$, and $\rho_T(p) = \sup \{ \rho_T(q) + 1 : q < p \}$.  The rank of $T$ is $\rho_T(\emptyset)$.

\begin{theorem}
\label{wftrees}
Suppose there is a countable transitive model $N$ of \ZFC containing a completely J\'onsson cardinal $\delta$ that has infinitely many measurables below it, $\alpha \in N$ is an ordinal, and $T$ is a well-founded sequential tree of rank $\alpha$.  Then there is an honest subcategory of $\mathcal A_\delta$ isomorphic to $T$.
\end{theorem}

\begin{proof}
For the purposes of this argument, for a natural number $n$ and a model $M$ in the language of set theory, let $\kappa_n^M$ denote the $n^{th}$ measurable cardinal in $M$, if it exists.  We show the following claim by induction on $\alpha$.  Let $\Phi(T,\alpha,M,\delta,f)$ stand for the assertion that:
\begin{enumerate}
\item $T$ is a well-founded sequential tree of rank $\alpha$.
\item $M \in N$ is a transitive model of $\ZF -$ Powerset + ``Every set has a well-ordering'' + ``$\delta$ is a completely J\'onsson cardinal with infinitely many measurables below it'' $+$ ``$V_{\delta+\alpha+1}$ exists''.
\item $f : \omega^2 \to \omega$ is an injection in $M$.
\end{enumerate}
We claim that if $\Phi(T,\alpha,M,\delta,f)$ holds,
then there is an assignment $p \mapsto M_p$ for $p \in T$ such that:
\begin{enumerate}
\item Each $M_p$ is a transtive set of rank $\delta$ in $M$, with $M_\emptyset = V^M_\delta$
\item If $p \leq q$, then there is an amenable $j : M_p \to M_q$ in $M$.
\item For all $p \in T$, $\kappa_{f(i,p(i))}^{M_p} < \kappa_{f(i,p(i))}^M$ for $i < \len(p)$, and $\kappa_{f(i,i')}^{M_p} = \kappa_{f(i,i')}^M$ for $(i,i') \notin p$.
\end{enumerate}
Let us first show why the claim implies that the generated subcategory is honest.  Suppose $p(i) \not= q(i)$, and let $n = f(i,q(i))$.  If there were an elementary embedding $j : M_p \to M_q$, then 
$$j(\kappa^{M_p}_n) = \kappa^{M_q}_n < \kappa_n^M = \kappa^{M_p}_n,$$
which is impossible since embeddings are monotonic on ordinals.

Suppose $\Phi(T,\alpha,M,\delta,f)$ holds, and the claim holds below $\alpha$.
 Without loss of generality, $\la n \ra \in T$ for each $n<\omega$, since every sequential tree of rank $\alpha$ can be embedded into one of this form.  For each $n$, let $T_n = \{ \la p(1),\dots,p(\len(p) -1) \ra : p \in T$ and $p(0) = n \}$.  Note that each $T_n$ is a well-founded sequential tree, and $\alpha = \sup_n (\rank(T_n)+1)$.  Let $\alpha_n = \rank(T_n)$.

Working in $M$, let $\theta = |V_{\delta+\alpha}|^+$.  We can take for each $n<\omega$, an elementary substructure $\frak A_n \prec H_\theta$ such that:
\begin{enumerate}
\item $| \frak A_n \cap \kappa_{f(0,n)}^M | < \kappa_{f(0,n)}^M$.
\item $| \frak A_n \cap \kappa_{f(i,i')}^M | = \kappa_{f(i,i')}^M$ for $(i,i') \not= (0,n)$.
\item $\delta \in \frak A_n$ and $| \frak A_n \cap \delta | = \delta$. 
\end{enumerate}
Let $M_n$ be the transitive collapse of $\frak A_n$.  Let $M_\emptyset = V_\delta^M$ and $M_{\la n \ra} = (V_\delta)^{M_n}$.  (Note that $M_{\la n \ra}$ is a rank initial segment of $M_n$, and the larger model does not satisfy the powerset axiom.)  Each $M_{\la n \ra}$ amenably embeds into $M_\emptyset$.  For each $n$, $\kappa_{f(0,n)}^{M_n} < \kappa_{f(0,n)}^M$ and $\kappa_{f(i,i')}^{M_n} = \kappa_{f(i,i')}^M$ for $(i,i') \not= (0,n)$.

Let $g : \omega^2 \to \omega$ be defined by $g(i,i') = f(i+1,i')$.  Then $\Phi(T_n,\alpha_n,M_n,\delta,g)$ holds for each $n$.  By the induction hypothesis, there is an assignment $p \mapsto M_p$, for $p \in T \setminus \{ \emptyset \}$, such that whenever $p(0) = n$,
\begin{enumerate}
\item $M_p$ is a transtive set of rank $\delta$ in $M_n$.
\item If $q \leq p$, then there is an amenable $j : M_q \to M_p$ in $M_n$.
\item For $1 \leq i < \len(p)$, 
$$\kappa_{g(i-1,p(i))}^{M_p} = \kappa_{f(i,p(i))}^{M_p} <  \kappa_{f(i,p(i))}^{M_{n}} = \kappa_{f(i,p(i))}^M,$$
and for a pair $(i,i') \notin p$, with $i \geq 1$,
$$\kappa_{g(i-1,i')}^{M_p} = \kappa_{f(i,i')}^{M_p} =  \kappa_{f(i,i')}^{M_{n}} = \kappa_{f(i,i')}^M.$$
\item For all $m<\omega$, $\kappa_{f(0,m)}^{M_p} = \kappa_{f(0,m)}^{M_n}$.
\end{enumerate}
Therefore, the induction hypothesis holds for $\alpha$.
\end{proof}

\subsection{Forcing outward}

The background material for this subsection can be found in \cite{MR2069032}.

\begin{definition*}[Woodin]
The (proper class) \emph{stationary tower forcing} $\mathbb P_\infty$ is the class of stationary sets ordered by $a \leq b$ if $\bigcup a \supseteq \bigcup b$ and $\{ x \cap \bigcup b : x \in a \} \subseteq b$.  If $\kappa$ is a strong limit cardinal, then $\mathbb P_{<\kappa} = \mathbb P_\infty \cap V_\kappa$.
\end{definition*}

\begin{theorem}[Woodin]
Suppose there is a proper class of completely J\'onsson cardinals, and $G \subseteq \mathbb P_\infty$ is generic over $V$.  Then there is an amenable elementary embedding $j : V \to V[G]$ with unboundedly many regular fixed points.
\end{theorem}

The above theorem immediately implies that if there is a countable transitive model $M_0$ of height $\delta$ with a proper class of completely J\'onsson cardinals, then $\mathcal A_\delta$ contains a subcategory isomorphic to the linear order $\omega$.  Given $M_n$, we can take a $\mathbb P_\infty$-generic $G_n$ over $M_n$, which gives us an amenable $j : M_n \to M_{n+1} = M_n[G_n]$.  In order to construct more complicated subcategories, we use the following to determine the action of the maps:

\begin{fact}
\label{los}
Suppose $G \subseteq \mathbb P_\infty$ and $j : V \to V[G]$ are as above.  Suppose $\varphi(x_0,\dots,x_n,y)$ is a formula,  $a_0,\dots,a_n,b \in V$, and $b$ is transitive.  Then 
$$V[G] \models \varphi(j(a_0),\dots,j(a_n),j[b]) \Leftrightarrow \{ z \subseteq b : V \models \varphi(a_0,\dots,a_n,z) \} \in G.$$
Consequently, $\kappa = \crit(j) \Leftrightarrow \{ z \subseteq \kappa : z \cap \kappa \in \kappa \} \in G$.
\end{fact}

To construct subcategories containing transfinite chains, we use our latitude in determining critical points so that the equivalence classes that make up the ordinals of the direct limit are all eventually ``frozen.''

\begin{lemma}
\label{omegalimit}
Suppose the following:
\begin{enumerate}
\item $\la \delta_n : n < \omega \ra$ is an increasing sequence of ordinals.
\item $\la M_n : n <\omega \ra$ is a sequence of transitive ZFC models, each of height $\delta = \sup_{n<\omega} \delta_n$.
\item $\la j_{m,n} : M_m \to M_n : m \leq n < \omega \ra$ is a commuting system of elementary embeddings.
\item For all $m \leq n$, $j_{m,n}(\delta_n)<\crit(j_{n,n+1}) $. 
\end{enumerate}
Then the direct limit is isomorphic to a transitive model $M_\omega \subseteq \bigcup_{n<\omega} M_n$. If each $j_{m,n}$ is amenable, then so is each direct limit map $j_{n,\omega}$, and thus $M_\omega = \bigcup_{n<\omega} M_n$.
\end{lemma}

\begin{proof}
Recall that the direct limit is defined as the set of equivalence classes $[n,x]$, where $x \in M_n$, and we put $[n,x] \sim [m,y]$ for $n \leq m$ when $j_{n,m}(x) = y$.  The membership relation is defined similarly: for $n\leq m$, $[n,x]$ ``$\in$'' $[m,y]$ when $j_{n,m}(x) \in y$.  
Suppose $\alpha < \delta$ and $n<\omega$.  Let $m \geq n$ be such that $\delta_m \geq \alpha$, and let $\beta = j_{n,m}(\alpha)$.  Then for all $k > m$, $\crit(j_{m,k})>\beta$.  Thus if $[k,\gamma] < [n,\alpha]$ in the direct limit, then there is some $\xi <\beta$ such that $[k,\gamma] \sim [m,\xi]$.  Thus the set of predecessors of any ordinal in the direct limit is isomorphic to an ordinal below $\delta$, so the direct limit can be identified with a transitive model $M_\omega$ of height $\delta$.  To show $M_\omega \subseteq \bigcup_n M_n$, let $x \in M_\omega$ and let $n,y$ be such that $x = j_{n,\omega}(y)$.  Let $m\geq n$ be such that $\rank(y) <\delta_m$.  $\crit(j_{m,\omega}) > \rank(j_{n,m}(y))$, so $x = j_{m,\omega} \circ j_{n,m}(y) = j_{n,m}(y) \in M_m$.

Now suppose each $j_{m,n}$ is amenable.  Let $x \in M_n$.  Let $m\geq n$ be such that $\rank(x) <\delta_m$.  The critical point of $j_{m,k}$ is greater than $j_{n,m}(\delta_m)$ for all $k>m$.  By amenability, $j_{n,m}[x] \in (V_{j_{n,m}(\delta_m)})^{M_m}$.  But $j_{n,\omega}[x] = j_{m,\omega}[j_{n,m}[x]] =  j_{m,\omega}(j_{n,m}[x]) \in M_\omega$.  \end{proof}

\begin{corollary}
\label{anycountable}
Suppose $M$ is a countable transitive model of \ZFC of height $\delta$ satisfying that there is a proper class of completely J\'onsson cardinals.  Let $\xi$ be any countable ordinal.  Then there is a subcategory of $\mathcal A_\delta$ isomorphic to $\xi$.
\end{corollary}

\begin{proof}
We argue by induction on $\xi$, using the following stronger hypothesis:  If $M_0$ is a countable transitive model of \ZFC + ``There is a proper class of completely J\'onsson cardinals'' of height $\delta$, and $\gamma <\delta$, then there is a linear system of amenable embeddings $\la j_{\alpha,\beta} : M_\alpha \to M_\beta : \alpha \leq \beta \leq \xi \ra$ with $\crit(j_{0,\xi}) > \gamma$, each $M_\alpha$ having the same ordinals $\delta$.  Suppose this is true for all $\xi' < \xi$.  If $\xi =\xi'+1$, then we may continue to $\xi$ by taking the final embedding $j_{\xi',\xi}$ to be a stationary tower embedding with large enough critical point.  Suppose $\xi$ is a limit ordinal, and let $\xi = \sup_{n<\omega} \xi_n$ and $\delta = \sup_{n<\omega} \delta_n$, where the $\xi_n$ and $\delta_n$ are increasing.  Let $\gamma<\delta$.  Inductively choose for each $n<\omega$, a chain of amenable embeddings $\la j_{\alpha,\beta} : M_\alpha \to M_\beta : \xi_n \leq \alpha \leq\beta\leq\xi_{n+1} \ra$, such that for all $m\leq n <\omega$, $\crit(j_{\xi_n,\xi_{n+1}}) > j_{\xi_m,\xi_n}(\delta_n)+\gamma$.  Then by Lemma \ref{omegalimit}, the direct limit $M_\xi$ is well-founded of height $\delta$, and each direct limit map $j_{\xi_n,\xi} : M_{\xi_n} \to M_\xi$ is amenable.  For $\alpha<\xi$, let $j_{\alpha,\xi} : M_\alpha \to M_\xi$ be $j_{\xi_n,\xi} \circ j_{\alpha,\xi_n}$, where $n$ is least such that $\xi_n \geq \alpha$.  For $\alpha<\beta<\xi$ coming from different intervals, let $j_{\alpha,\beta} = j_{\xi_m,\beta} \circ j_{\xi_{m-1},\xi_m} \circ \dots \circ j_{\xi_n,\xi_{n+1}} \circ j_{\alpha,\xi_n}$, where $n$ is least such $\xi_n \geq \alpha$ and $m$ is greatest such that $\xi_m \leq \beta$.  These are amenable since they are the composition of amenable maps.  It is easy to check that $\crit(j_{0,\xi}) > \gamma$.
\end{proof}

The countability of $\xi$ above is optimal:

\begin{theorem}
\label{nolongchains}
For any ordinal $\delta$, $\mathcal E_\delta$ contains no subcategory isomorphic to $\delta^+$.  Moreover, if $T$ is a tree of height $\delta^+$ such that for every $t \in T$ and every $\alpha<\delta^+$, there is $s \geq t$ of height $\geq \alpha$, and $T$ is isomorphic to a subcategory of $\mathcal E_\delta$, then forcing with $T$ collapses $\delta^+$.
\end{theorem}
\begin{proof}
Suppose $\{ M_\alpha : \alpha < \delta^+ \}$ is set of objects in $\mathcal E_\delta$, and each pair $\alpha < \beta$ is assigned an embedding $j_{\alpha,\beta} : M_\alpha \to M_\beta$, such that for $\alpha < \beta < \gamma$, $j_{\alpha,\gamma} = j_{\beta,\gamma} \circ j_{\alpha,\beta}$.

We select continuous increasing sequences of ordinals, $\la \xi_i : i <\delta^+ \ra$, and subsets of $\delta$, $\la F_i : i < \delta^+\ra$. 
Let $\xi_0 = 0$ and let $F_0$ be the set of ordinals that are fixed by every map $j_{\alpha,\beta}$.  Given $\xi_i$, the sequence $\la j_{\xi_i,\beta}(\alpha) : \beta < \delta^+ \ra$ is nondecreasing.  Thus there is $\xi_{i+1}<\delta^+$ such that $j_{\xi_i,\xi_{i+1}}(\alpha) = j_{\xi_i,\beta}(\alpha)$ for all $\beta \geq \xi_{i+1}$, and all $\alpha<\delta$. Let $F_{i+1} = j_{\xi_i,\xi_{i+1}}[\delta]$.  Take unions at limits.  By induction, the sets $F_i$ are increasing, and for all $i<\delta^+$, $j_{\xi_i,\beta}(\alpha) = \alpha$ for $\alpha \in F_i$ and $\beta\geq\xi_i$.

Let $\gamma < \delta^+$ be the point at which the sets $F_i$ stabilize.  We claim that for all $i > \gamma$, $j_{\xi_\gamma,i}$ is the identity map and thus $M_i = M_{\xi_\gamma}$.  Otherwise, let $\alpha$ and $\eta$ be such that $j_{\xi_\gamma,\eta}(\alpha) > \alpha$. 
But then $j_{\xi_\gamma,\xi_{\gamma+1}}(\alpha) \in F_{\gamma+1} \setminus F_\gamma$,
a contradiction.  

To show the claim about $\delta^+$-trees, suppose on the contrary $T$ is a tree satisfying the hypothesis.  
If $b$ is $T$-generic branch over $V$, then $b$ has length $(\delta^+)^V$.  If $\delta^+$ is preserved, then in the extension, $b$ gives a chain of length $\delta^+$ in $\mathcal E_\delta$, contradicting what we have shown.
\end{proof}

To further explore the structure of these categories, we use countable transitive models $M \models \ZFC$ of some height $\delta$ which is the supremum of an increasing sequence $\la \delta_n : n<\omega \ra$ such that:
\begin{enumerate}
\item Each $\delta_n$ is a completely J\'onsson cardinal in $M$.
\item For each $n$, $V_{\delta_n}^M \prec M$.
\end{enumerate}
Let us call these \emph{good models}, and let us say an ordinal $\delta$ is good if there exists a good model of height $\delta$.  The existence of a good model follows from the existence a measurable cardinal.  Note that if $M$ is good, $N$ has the same ordinals, and $j : M \to N$ is elementary, then $\la j(\delta_n) : n <\omega \ra$ witnesses that $N$ is good.
Over good models, we can build generics for $\mathbb P_\infty$ in a way that allows precise control over what happens to the cardinals $\delta_n$:
  
\begin{lemma}
\label{controltower}
Let $f : \omega \to \omega$ be any increasing function, let $M$ be good with witnessing sequence $\la \delta_n : n < \omega \ra$, and let $\kappa_0<\kappa_1$ be regular cardinals of $M$ below $\delta_{f(0)}$.
There is $G \subseteq \mathbb P_\infty^M$ generic over $M$ such that if $j$ is the associated generic embedding, then
\begin{enumerate}
\item $\kappa_0 = \crit(j)$ and $j(\kappa_0) \geq \kappa_1$.
\item For all $n$, $\delta_{f(n)}$ is a fixed point of $j$.
\item For all $n$, if $f(n) < m < f(n+1)$,  
then $\delta_m$ has cardinality $\delta_{f(n)}$ in $M[G]$.
 \end{enumerate}
 \end{lemma}
 
\begin{proof}
Let $\la p_i : i < \omega \ra$ enumerate $M$, such that for all $n$, $p_n \in V_{\delta_n}$.  For each dense class $D \subseteq \mathbb P_\infty^M$ definable in $M$ from a parameter $p$, if $p \in V_{\delta_n}$, then by elementarity, $D \cap V_{\delta_n}$ is dense in $\mathbb P_{<\delta_n}$.  Let $n \mapsto \la n_0,n_1 \ra$ be the G\"odel pairing function, and let $\la \varphi_n(x,y) : n < \omega \ra$ enumerate the formulas in two free variables.  Now enumerate the dense subclasses of $\mathbb P_\infty$ as follows.  For $n< \omega$, if $\{ x : \varphi_{n_0}(x,p_{n_1}) \}$ is dense, let $D_n$ be this class, and otherwise let $D_n = \mathbb P_\infty$.

Let $b = \{ z \subseteq \kappa_1 : z \cap \kappa_0  \in \kappa_0, \ot(z) \leq \kappa_0 \},$
which by Fact \ref{los} forces that $\kappa_1 = \ot(j[\kappa_1]) \leq j(\kappa_0)$.  Then take $b' \leq b$ in $D_0 \cap V_{\delta_{f(0)}}$.  Let $a_0 = \{ z \subseteq V_{\delta_{f(0)}} : z \cap \bigcup b' \in b', \ot(z \cap \delta_{f(0)} )= \delta_{f(0)} \}$, which forces $\delta_{f(0)}$ to be fixed.

We choose inductively a sequence $a_0 \geq a_1 \geq a_2 \geq \dots$ such that for all $n$:
\begin{enumerate}
\item $a_{n}$ is below something in $D_{n}$.
\item $a_{n} \subseteq \{ z \subseteq V_{\delta_{f(n)}} : \ot(z) = \delta_{f(n)} \}$.
\item $a_{n+1}$ forces $\delta_{f(n+1)-1}< j(\delta_{f(n)}^+)<\delta_{f(n+1)}$.
\end{enumerate}
The last condition ensures that if $f(n)<m<f(n+1)$, then $\delta_m$ is collapsed to $\delta_{f(n)}$ in $M[G]$. 
Given $a_n$, take a cardinal $\kappa$, $\delta_{f(n+1)-1}< \kappa<\delta_{f(n+1)}$, and let $a_n'$ be the stationary set $\{ z \subseteq V_\kappa : z \cap \bigcup a_n \in a_n, |z| = \delta_{f(n)} \}$.  By Fact \ref{los}, this forces that $\kappa = \ot(j[\kappa]) < j(\delta_{f(n)}^+)$.   Then extend to $a_n'' \in D_{n+1} \cap V_{\delta_{f(n+1)}}$.  Then let $a_{n+1} = \{ z \subseteq V_{\delta_{f(n+1)}} : z \cap \bigcup a_n'' \in a_n'', \ot(z \cap \delta_{f(n+1)}) = \delta_{f(n+1)}\}$. 
\end{proof}

\begin{theorem}
\label{aronszajn}
If $\delta$ is good, then $\mathcal A_\delta$ contains an honest subcategory isomorphic to an Aronszajn tree.
\end{theorem}

\begin{proof}
Let us first describe our mechanism for splitting into incompatible nodes.  Suppose $M$ is a good model of height $\delta$ with witnessing sequence $\la \delta_n : n < \omega \ra$.  Let us partition the natural numbers into blocks $B_0,B_1,B_2,\dots$ as follows.  Let $B_0 = \emptyset$, and given $B_{n-1}$, let $B_n$ be the next $n$ numbers above $B_{n-1}$.  In other words, $B_n = \{ \frac{n(n-1)}{2},\frac{n(n-1)}{2} + 1, \dots,\frac{n(n-1)}{2} + n -1 \}$ for $n>0$.  Let $\la B_n(i) : i < n \ra$ enumerate block $B_n$ in increasing order.
Using Lemma \ref{controltower}, for each $n$, we can take a $\mathbb P_\infty$-generic $G_n$ over $M$ with associated embedding $j_n$ such that:
\begin{enumerate}
\item $\crit(j_n) \geq \delta_n$.% and $j_n(\delta_n) = \delta_{n+1}$.
\item For sufficiently large $m$, $\delta_{B_m(n)}$ is fixed by $j_n$.
\item For sufficiently large $m$, if $B_{m}(n)<i<B_{m+1}(n)$, then $\delta_i$ has cardinality $\delta_{B_{m}(n)}$ in $M[G_n]$.
\end{enumerate}
This ensures that if $n_0 < n_1$, then there is no model $N$ of ZF with the same class of ordinals $\delta$ containing both $M[G_{n_0}]$ and $M[G_{n_1}]$.  For if there were such an $N$, then in $N$, for sufficiently large $m$, $|\delta_{B_m(n_0)}| = |\delta_{B_m(n_1)}| = |\delta_{B_{m+1}(n_0)}|$, and so $N$ would have a largest cardinal.

Let $M_0$ be a good model of height $\delta$, with witnessing sequence $\la \delta_n : n <\omega \ra$. 
We will build an $\omega_1$-tree $T \subseteq \,^{<\omega_1}\omega$ by induction.  Each $\sigma \in T$ will be assigned a good model $M_\sigma$ of height $\delta$, and for pairs $\sigma \subseteq \tau$ in $T$, we will assign an amenable embedding $j_{\sigma,\tau} : M_\sigma \to M_\tau$, such that if $\sigma \subseteq \tau \subseteq \upsilon$, then $j_{\sigma,\upsilon} = j_{\tau,\upsilon}\circ j_{\sigma,\tau}$.  When $\sigma \perp \tau$, we will arrange that there is not any ZF model $N$ of height $\delta$ such that $M_\sigma,M_\tau \subseteq N$, which guarantees that there is no arrow between $M_\sigma$ and $M_\tau$ in the category $\mathcal A_\delta$.  This makes the tree of models an honest subcategory.  Theorem \ref{nolongchains} will ensure that $T$ is Aronszajn.

We build $T$ as a union of trees $T_\alpha$ for $\alpha < \omega_1$, where $T_\alpha$ has height $\alpha+1$, and $T_\beta$ end-extends $T_\beta$ for $\beta>\alpha$.  Along the way, we build an increasing sequence of injective functors $F_\alpha : T_\alpha \to \mathcal A_\delta$.  Put $T_0 = \{\emptyset\}$ and $F_0(\emptyset) = M_0$.
Given $T_\alpha$ and $F_\alpha$, form $T_{\alpha+1}$ by adding $\sigma ^\smallfrown \la n \ra$ for each top node $\sigma$ and each $n < \omega$.  For such top nodes $\sigma$, take an embedding $j_{\sigma,\sigma^\smallfrown\la n\ra} : M_\sigma \to M_{\sigma,\sigma^\smallfrown\la n\ra} = M_\sigma[G_n]$ as above, so that no two of them amalgamate.  For $\tau \subseteq \sigma$, let $j_{\tau,\sigma^\smallfrown\la n\ra} = j_{\sigma,\sigma^\smallfrown\la n\ra} \circ j_{\tau,\sigma}$.

Suppose $\lambda$ is a countable limit ordinal and we have constructed $T_\alpha$ and $F_\alpha$ for $\alpha <\lambda$.  Let $T_\lambda'$ and $F_\lambda'$ be the unions of the previously constructed trees and functors.  We must build the top level.
Let us assume the following inductive hypothesis:
for each $\alpha <\lambda$, each $\sigma \in T_\alpha$, and each $n<\omega$, 
there is $\tau \supseteq \sigma$ at the top level of $T_\alpha$ such that 
$\crit(j_{\sigma,\tau}) \geq \delta_n$.
To construct $T_\lambda$, use the induction hypothesis to choose for each node $\sigma\in T_\lambda'$ and each $n<\omega$, an increasing sequence $\la \tau_i : i < \omega \ra \subseteq T_\lambda'$ such that $\tau_0 = \sigma$, $\sup_i \len(\tau_i) = \lambda$, and for each $i < \omega$, 
$$\crit(j_{\tau_i,\tau_{i+1}})> \max \{ \delta_n,\delta_i,j_{\tau_0,\tau_i}(\delta_i),j_{\tau_1,\tau_i}(\delta_i),\dots,j_{\tau_{i-1},\tau_i}(\delta_i) \}.$$
Let $\tau(\sigma,n) = \bigcup_i \tau_i$.  
By Lemma \ref{omegalimit}, the direct limit of the system $\la j_{\tau_i,\tau_m} : M_{\tau_i} \to M_{\tau_m} : i \leq m < \omega \ra$ yields a model $M_{\tau(\sigma,n)}$ of height $\delta$ such that the direct limit maps $j_{\tau_i, \tau(\sigma,n)} : M_{\tau_i} \to M_{\tau(\sigma,n)}$ are amenable.  If $\sigma \subseteq \upsilon \subseteq \tau(\sigma,n)$, then the map $j_{\upsilon,\tau(\sigma,n)}$ is given by $j_{\tau_i,\tau(\sigma,n)} \circ j_{\upsilon,\tau_i}$ for any $i$ such that $\tau_i \supseteq \upsilon$.
Let the top level of $T_\lambda$ consist of all such $\tau(\sigma,n)$, and let the functor $F_\lambda$ be defined as above.  By construction, the inductive hypothesis is preserved.
\end{proof}

Similar constructions yield other kinds of uncountable trees as honest subcategories of $\mathcal A_\delta$. For example, we can build a copy of $^{<\omega}\omega$ where we also require that if $\sigma$ has length $n$, then for every initial segment $\tau \subseteq \sigma$ and every $m<\omega$, $\crit(j_{\sigma,\sigma^\smallfrown \la m \ra}) > j_{\tau,\sigma}(\delta_n)$.  Then \emph{every} branch yields a nice direct limit.  An induction as in Corollary \ref{anycountable}  allows this to be extended to any countable height.  Let us record this as:
\begin{prop}If $\delta$ is good, then for every countable ordinal $\alpha$, there is an honest subcategory of $\mathcal A_\delta$ isomorphic to the complete binary tree of height $\alpha$.
\end{prop}

\subsection{Ill-founded subcategories}

In order to construct ill-founded subcategories of $\mathcal A_\delta$, we pass to models of enough set theory that are well-founded beyond $\delta$ and contain some chosen member of $\mathcal A_\delta$, but ultimately have ill-founded $\omega_1$.  One way of achieving this is via Barwise compactness \cite{MR0406760}.  Another is via the following:

\begin{fact}[see \cite{MR2768692}, Section 2.6]
Suppose $r$ is a real.  
Let $G$ be generic for $\p(\omega_1)/\ns$ over $L[r]$.  Let $j : L[r] \to N = \Ult(L[r],G)$ be the ultrapower embedding.  Then $N$ is well-founded up to $\omega_2^{L[r]}$, but $\omega_1^N$ is ill-founded.
\end{fact}

We can decompose each ordinal in such a model into a sum of purely ill-founded and a purely well-founded part.  Let $N$ be any countable ill-founded model of $\ZF -$Powerset that has standard $\omega$.  Let $\xi$ be the ordertype of the set of well-founded ordinals of $N$.  H. Friedman \cite{MR0347599} showed that the ordertype of the ordinals of $N$ is $\xi + (\mathbb Q \times \xi)$.  Thus every $N$-ordinal $\alpha$ can be written as $\alpha = i(\alpha) + w(\alpha)$, where $w(\alpha) <\xi$, and $i(\alpha)$ is either 0 or is represented in the isomorphism as $(q,0)$, where $q \in \mathbb Q$.

We apply this trick to construct honest subcategories of $\mathcal A_\delta$ isomorphic to three canonical ill-founded partial orders: the real numbers, the universal countable pseudotree (to be defined below), and the reverse-ordered complete binary tree.  We present them in increasing order of the consistency strength of the hypothesis employed.

\begin{theorem}
\label{R}
Suppose there is a countable transitive model of height $\delta$ of \ZFC + ``There is a proper class of completely J\'onsson cardinals.''  Then there is a subcategory of $\mathcal A_\delta$ isomorphic to the linear order $\mathbb R$.
\end{theorem}

\begin{proof}
Suppose $M$ is a countable transitive model of height $\delta$ of \ZFC + ``There is a proper class of completely J\'onsson cardinals'' and $r_M$ is a real coding $M$.  Let $N$ be a countable model of $\ZFC-$Powerset with standard $\omega$ and ill-founded $\omega_1$, such that $r_M \in N$ and the conclusion of Corollary \ref{anycountable} holds in $N$.

Fix some $\zeta$ that $N$ thinks is a countable ordinal but is really ill-founded.  There is a subcategory of $\mathcal A_\delta$ isomorphic to the linear order $\{ \alpha : N \models \alpha < \zeta \}$.  If we fix an increasing sequence $\la \delta_n : n < \omega \ra$ cofinal in $\delta$, then we can modify the construction in Corollary \ref{anycountable} slightly to require that if $\kappa = \crit(j_{\alpha,\alpha+1})$ and $\delta_n \leq \kappa <\delta_{n+1}$, then $j_{\alpha,\alpha+1}(\kappa) \geq \delta_{n+1}$.

Thus there is a subcategory of $\mathcal A_\delta$ isomorphic to the rationals, which has as its objects all models indexed by $i(\alpha)$ for $0<i(\alpha)<\zeta$.  Let us write this as $\la j_{a,b} : M_a \to M_b : a \leq b \in \mathbb Q \ra$.  We have that for rationals $a<b$, if $\kappa = \crit(j_{a,b})$ and $\delta_n \leq \kappa < \delta_{n+1}$, then $j_{a,b}(\kappa) \geq \delta_{n+1}$.  Now we argue that we can fill in the Dedekind cuts with direct limits and get a larger subcategory of $\mathcal A_\delta$.

Suppose $r \in \mathbb R \setminus \mathbb Q$. 
Let $M_r$ be the direct limit and let $j_{a,r}$ be the direct limit map for rationals $a<r$.  Let $b > r$ be rational.  For $x \in M_r$, there is a rational $a<r$ and $y \in M_a$ such that $x = j_{a,r}(y)$.  Let $j_{r,b}(x) = j_{a,b}(y)$ for any such $a,y$.  The choice of $a,y$ does not matter, since if $x = j_{a,r}(y) = j_{a',r}(y')$ for $a < a'$, then $y' = j_{a,a'}(y)$, and $j_{a,b}(y) = j_{a',b}(y')$.  Furthermore, if $b,b'$ are rational such that $r<b<b'$, then $j_{r,b'} = j_{b,b'} \circ j_{r,b}$.  This is because if $x \in M_r$ and $a,y$ are such that $j_{a,r}(y) = x$, then 
$$j_{r,b'}(x) = j_{a,b'}(y) = j_{b,b'} \circ j_{a,b}(y)
= j_{b,b'} \circ j_{r,b}(x).$$
 If $M_r \models \varphi(x_0,\dots,x_{n-1})$, then there are a rational $a<r$ and $y_0,\dots,y_{n-1} \in M_a$ such that $j_{a,r}(y_m) = x_m$ for $m<n$, and so $M_b \models \varphi(j_{a,b}(y_0),\dots,j_{a,b}(y_{n-1}))$.  Thus $j_{r,b}$ is elementary, and so $M_r$ is well-founded and has height $\delta$.  For irrational numbers $r<s$,
% and $x \in M_r$, 
let $j_{r,s} = j_{b,s} \circ j_{r,b}$ for any rational $b \in (r,s)$.  The choice of $b$ does not matter, since for $b,b'$ rational such that $r < b < b' < s$, we have:
$$j_{b,s} \circ j_{r,b} = j_{b',s} \circ j_{b,b'} \circ j_{r,b} = j_{b',s} \circ j_{r,b'}.$$
To verify the commutativity of the system, let $r<s<t$ be reals and let $x \in M_r$.  
Let $b,c$ be rational such that $r<b<s<c<t$.  Then
$$j_{r,t}(x) = j_{b,t} \circ j_{r,b}(x) = j_{c,t} \circ j_{b,c} \circ j_{r,b}(x).$$
But by definition, $j_{b,c}(z) = j_{s,c} \circ j_{b,s}(z)$ for any $z \in M_b$.  Thus: 
$$j_{r,t}(x) = j_{c,t} \circ \ j_{s,c} \circ j_{b,s} \circ j_{r,b}(x) = j_{s,t} \circ j_{r,s}(x).$$

To show that the maps are amenable, first suppose $r$ is irrational, $a<r$ is rational, and $\gamma < \delta$.  Let $\la a_n : n < \omega \ra$ be an increasing sequence of rationals converging to $r$, with $a_0 = a$.  For $n<\omega$, let $\kappa_n = \min \{ \crit(j_{a_n,a_m}) : n<m < \omega \}$.  By passing to a subsequence if necessary, we may assume that for all $n$, $\kappa_n = \crit(j_{a_n,a_{n+1}})$, and $\kappa_n \leq \kappa_{n+1}$.  Now it cannot happen that for infinitely many $n$ and all $m > n$, $j_{a_n,a_m}(\kappa_n) > \kappa_m$, because then the direct limit would be ill-founded.  Thus there is some $n^*$ such that for all $n \geq n^*$, there is $m>n$ such that $j_{a_n,a_m}(\kappa_n) \leq \kappa_m$.  By our extra requirement on how high the critical points are sent, this means that for large enough $n$, $\kappa_n > \gamma$.  Therefore, for a large enough $n$, $j_{a,r}[\gamma] = j_{a_n,r}(j_{a,a_n}[\gamma]) \in M_r$.

Now suppose $r$ is irrational, $b>r$ is rational, and $\gamma<\delta$.  By the previous paragraph, there is some rational $a<r$ such that $\crit(j_{a,r}) > \gamma$.  Thus $j_{r,b}[\gamma] = j_{a,b}[\gamma] \in M_b$.  Finally, if $r<s$ are irrational, the map $j_{r,s}$ is amenable since it is the composition of two amenable maps $j_{b,s} \circ j_{r,b}$, for any rational $b \in (r,s)$.
\end{proof}

We note also that $M_r = \bigcup_{a<r} M_a$ for any $r$.  This is true for rational $r$ by the construction in the proof of Lemma \ref{anycountable}, since all limit models are direct limits.  If $r$ is irrational and $x \in M_r$, then $x \in M_b$ for all rational $b > r$.  If $x \notin \bigcup_{a<r} M_a$, then $N$ thinks there is some least ordinal $\gamma$ such that $x \in M_\gamma$, which must be a successor ordinal of $N$.  Thus $x \notin M_{i(\gamma)}$, and $i(\gamma)$ corresponds to a rational number greater than $r$.  This contradiction shows that $x \in M_a$ for some $a<r$.

Using stationary tower forcing as the basic building block, we have constructed various (well-founded) trees and linear orders as honest subcategories of $\mathcal A_\delta$.  A natural generalization of these two kinds of partial orders is the notion of a \emph{pseudotree}, which is simply a partial order which is linear below any given element.  By the well-known universality of $\mathbb Q$, all countable linear orders appear in $\mathcal A_\delta$ for appropriate $\delta$.  In fact, this generalizes to pseudotrees:

\begin{theorem}
\label{pseudo}
If $\delta$ is good, then every countable pseudotree is isomorphic to an honest subcategory of $\mathcal A_\delta$.
\end{theorem}

To prove this, we first construct a certain countable pseudotree which contains every other countable pseudotree as a substructure.  Let $T_{\mathbb Q}$ be the collection of all partial functions $f : \mathbb Q \to \omega$ such that:
\begin{enumerate}
\item $\dom f$ is a proper initial segment of $\mathbb Q$.
\item If $f \not= \emptyset$, then there is a finite sequence $-\infty = q_0 < q_1 < \dots < q_n = \max(\dom f)$ such that for $i <n$, $f \restriction (q_i,q_{i+1}]$ is constant.
\end{enumerate}
We put $f \leq g$ when $f \subseteq g$.  Notice that $T_{\mathbb Q}$ satisfies the following axiom set, which we call DPM for ``Dense Pseudotrees with Meets'':
\begin{enumerate}
\item It's a pseudotree with a least element 0.
\item Every two elements $f,g$ have an infimum $f \wedge g$.
\item Infinite Splitting: For all $g,f_0,\dots,f_{n-1}$ such that $f_i \wedge f_j = g$ for $i<j < n$, there is $f_n > g$ such that $f_i \wedge f_n = g$ for $i< n$.
\item Density: If $g<f$, there is $h$ such that $g<h<f$.
\end{enumerate}
To be more precise, we take the language to be $\{ 0, \wedge, \leq \}$, where 0 is a constant, $\wedge$ is a binary function, and $a<b$ means $a \leq b$ and $a \not= b$.  The third axiom is actually an infinite scheme.  Let PM be the system consisting of only the first two axioms.  Note that in the theory PM, $\leq$ is definable in terms of $\wedge$: $p \leq q$ iff $p \wedge q = p$.  Thus to verify that a map from one model of PM to another is an embedding, it suffices to show that it is an injection that is a homomorphism in the restricted language $\{0,\wedge\}$.

\begin{lemma}
\label{universal}
Suppose $T$ satisfies PM, $Q$ satisfies DPM, $a \in T$, $S$ is a finite substructure of $T$, and $\pi : S \to Q$ is an embedding.  Then there is a finite substructure $S'$ of $T$ containing $S \cup \{ a \}$ and an embedding $\pi' : S' \to Q$ extending $\pi$.
\end{lemma}

\begin{proof}
Let us assume $a \notin S$.  Let $b = \max \{ a \wedge s : s \in S \}$.  Note that $a \wedge s = b \wedge s$ for all $s \in S$ by the maximality of $b$.  First we claim that $S \cup \{ b \}$ and $S \cup \{ a,b \}$ are both closed under the operation $\wedge$.  Let $s_0 = \inf \{ s \in S : s \geq b \}$.  Since $T$ is linearly ordered below $s_0$, for all $s \in S$,
$$b \wedge s =
\begin{cases} 
b 			&\text{ if } b \leq s_0 \wedge s \\
s_0 \wedge s 	&\text{ if } s_0 \wedge s \leq b. \\
\end{cases}$$
Thus for all $s \in S$, $b \wedge s \in S \cup \{ b \}$.
Further, if $b \leq s \wedge s_0$, then $b \leq a \wedge s$.
 Thus for all $s \in S$, either $a \wedge s = b$, or $b > a \wedge s$ and so $b > s \wedge s_0$, which implies $a \wedge s = b \wedge s = s \wedge s_0 \in S$.
Therefore, $S \cup \{ a,b \}$ is also closed under $\wedge$.

Let us first extend $\pi$ to an embedding $\pi' : S \cup \{b \} \to Q$.  If $b \notin S$, let $s_1 \in S$ be the maximum element below $b$.  Use Density to pick some $b^* \in Q$ such that $\pi(s_1)<b^*<\pi(s_0)$, and let $\pi'(b) = b^*$.  Suppose $s \in S$. If $s_0 \leq s$, then $\pi'(b \wedge s) = \pi'(b) = b^* = b^* \wedge \pi(s)$, since $\pi(s) \geq \pi(s_0) > b^*$.  If $s_0 \nleq s$, then $s_0 \wedge s \leq s_1$ by the minimality of $s_0$, and $\pi'(b \wedge s) = \pi(s_0 \wedge s) = \pi(s_0) \wedge \pi(s) = b^* \wedge \pi(s_0) \wedge \pi(s) =  b^* \wedge \pi(s)$, since $\pi(s) \wedge \pi(s_0) \leq \pi(s_1) < b^*$.

Now if $b$ is not maximal in $S \cup \{b \}$, let $\{ c_0,\dots,c_{n-1} \}$ be the set of minimal elements above $b$ in $S$, which implies that $c_i \wedge c_j = b$ for $i < j < n$.  Use Infinite Splitting to find $a^* \in Q$ such that $a^* > b^*$ and, if the $c_i$ are defined, $a^* \wedge \pi(c_i) = b^*$ for $i< n$.  We claim that for all $s \in S$, $a^* \wedge \pi(s) = b^* \wedge \pi(s)$.  For if $s = b$,  $a^* \wedge \pi(b) = a^* \wedge b^* = b^* = b^* \wedge b^*$.  If $s > b$, then $s \geq c_i$ for some $i < n$, so $a^* \wedge \pi(s) = b^*$.  If $s \ngeq b$, then $\pi(s) \wedge b^* < b^* < a^*$, so also $\pi(s) \wedge a^* < b^*$, and thus $a^* \wedge \pi(s) = b^* \wedge \pi(s)$.  Thus for all $s \in S$, $\pi'(a \wedge s) = \pi'(b \wedge s) = b^* \wedge \pi(s) = a^* \wedge \pi(s)$.  Therefore, we may define the desired extension $\pi''$ of $\pi'$ by putting $\pi''(a) = a^*$.
\end{proof}

\begin{corollary}
Any countable model of PM can be embedded into any model of DPM, and any two countable models of DPM are isomorphic.  
\end{corollary}
\begin{proof}
Suppose $S \models$ PM, $T \models$ DPM, and $S$ is countable.  Let $\la s_i : i < \omega \ra$ enumerate $S$. Using Lemma \ref{universal}, we can build increasing sequences of finite substructures $S_n \subseteq S$ and embeddings $\pi_n : S_n \to T$ for $n<\omega$, such that $S_n \supseteq \{ s_i : i < n \}$.  Then $\pi = \bigcup_{n<\omega} \pi_n$ is an embedding of $S$ into $T$.  If we additionally assume that $T = \{ t_i : i < \omega \}$ and $S \models$ DPM, then given $\pi_n : S_n \to T$, we can find a finite substructure $T_n \supseteq \ran \pi_n \cup \{ t_n \}$ and an embedding $\sigma_n : T_n \to S$ extending $\pi_n^{-1}$.  We then take $S_{n+1} \supseteq \ran \sigma_n \cup \{ s_n \}$ and $\pi_{n+1}$ extending $\sigma_n^{-1}$.   Then $\pi = \bigcup_{n<\omega} \pi_n$ is an isomorphism.
\end{proof}

\begin{lemma}
Every pseudotree is a substructure of a model of PM.
\end{lemma}

\begin{proof}
Let $T$ be a pseudotree.  $T$ is isomorphic to the collection of its subsets of the form $S_t = \{ x \in T : x \leq t \}$, ordered by inclusion.  Extend this to a collection $T^*$ by adding $S_{t_0} \cap S_{t_1}$ for every two $t_0,t_1\in T$, and $\emptyset$ if $T$ does not already have a minimal element.  We claim that $T^*$ is closed under finite intersection.  Let $t_0,t_1,t_2,t_3 \in T$.
The sets $S_{t_0} \cap S_{t_i}$ for $i < 4$ are initial segments of the linearly ordered set $S_{t_0}$, and so are $\subseteq$-comparable.  If $j$ is such that $S_{t_0} \cap S_{t_j}$ is smallest, then $S_{t_0} \cap S_{t_j} \subseteq S_{t_i}$ for all $i<4$.
To check that $T^*$ is a pseudotree, suppose $S_0,S_1,S_2 \in T^*$ and $S_0,S_1 \subseteq S_2$.  Again, $S_0$ and $S_1$ are initial segments of a linear order, so one is contained in the other.
\end{proof}

\begin{corollary}
Every countable pseudotree can be embedded into $T_{\mathbb Q}$.
\end{corollary}

Now we argue that if $\delta$ is a good ordinal, there is a countable honest subcategory of $\mathcal A_\delta$ that is a partial order satisfying DPM.  Suppose $M$ is a good model of height $\delta$ and $r_M$ is a real coding $M$.  Let $N$ be a countable model of $\ZFC -$ Powerset with standard $\omega$ and ill-founded $\omega_1$, such that $r_M \in N$ and $N$ satisfies the conclusion of Theorem \ref{aronszajn}.  In the real universe, the object $T$ that $N$ thinks is an Aronszajn tree isomorphic to an honest subcategory of $\mathcal A_\delta^N$ is actually a countable pseudotree isomorphic to an honest subcategory of $\mathcal A_\delta$, since $\delta$ is in the well-founded part of $N$, and because the non-amalgamability of models corresponding to incompatible nodes is witnessed in an absolute way.  We will find a substructure of $T$ that satisfies DPM.

As before, each $N$-ordinal $\alpha$ has a decomposition into ill-founded and well-founded parts, $\alpha = i(\alpha)+w(\alpha)$, and the collection $\{ i(\alpha) : \alpha < \omega_1^N \}$ is a dense linear order with a left endpoint.
Let $S \subseteq T$ be the collection of nodes at levels $i(\alpha)$ for $\alpha < \omega_1^N$.  Let us verify each axiom of DPM.  We have already argued that it's a pseudotree, and we keep the root node.  For infima, if $s,t \in S$ are incomparable, then $N$ thinks that there is a least ordinal $\alpha$ such that $s(\alpha) \not= t(\alpha)$.  Then $s \restriction i(\alpha) = t \restriction i(\alpha)$ is the infimum of $s$ and $t$ in $S$, since $i(\alpha)<i(\beta)$ implies $\alpha<\beta$.  To verify Density, just use the fact that the collection of levels of $S$ is a dense linear order.
To verify Infinite Splitting, let $g,f_0,\dots,f_{n-1}$ be as in the hypothesis, so that $g = f_k \wedge_S f_m$ for $k<m<n$.  Let $\gamma = \dom g$.  We may select any $f_n\in S$ such that $f_n \restriction \gamma = g$ and $f_n(\gamma) \not= f_m(\gamma)$ for all $m<n$.  This completes the proof of Theorem \ref{pseudo}.

\begin{remark}
As pointed out by the referee, Theorem \ref{pseudo} has the following generalization.  If we construct the tree of models and embeddings in the proof of Theorem \ref{aronszajn} with the additional requirement that the critical points are always sent sufficiently high, as in the proof of Theorem \ref{R}, then we can take a Dedekind completion of our universal countable pseudotree and obtain an isomorphic honest subcategory of $\mathcal A_\delta$.  Such a pseudotree will be universal for the class of pseudotrees possessing a countable dense subset.
\end{remark}

\begin{theorem}
If there is a set of measurable cardinals of ordertype $\omega+1$, then there is a countable ordinal $\delta$ such that $\mathcal A_\delta$ contains an honest subcategory isomorphic to the reverse-ordered complete binary tree of height $\omega$.
\end{theorem}

\begin{proof}
Let $\kappa$ be the $\omega+1^{st}$ measurable cardinal, with normal measure $\mathcal U$.  Let $M$ be a countable transitive set such that there is an elementary embedding $\sigma : M \to (H_\theta,\in,\mathcal U)$, for $\theta>2^\kappa$ regular.  Let $\kappa^*,\mathcal U^* \in M$ be such that $\sigma(\kappa^*) = \kappa$ and $\sigma(\mathcal U^*)=\mathcal U$.  Then $(M,\mathcal U^*)$ is iterable.  (See, for example, Lemma 19.11 of \cite{MR1994835}.)

Let $r_M$ be a real coding $(M,\mathcal U^*)$, and let $N$ be a countable model of $\ZFC - $ Powerset with standard $\omega$ and ill-founded $\omega_1$, such that $r_M \in N$, and satisfying that $(M,\mathcal U^*)$ is iterable.  Let $\delta < \kappa^*$ be completely J\'onsson in $M$ and above the first $\omega$ measurables of $M$.  Working in $N$, iterate $(M,\mathcal U^*)$ $\omega_1^N$-many times, and let $j : M \to M_{\omega_1^N}$ be the iterated ultrapower map.  Then $j$ is the identity below $\kappa^*$, and $j(\kappa^*) = \omega_1^N$.  Then take in $N$ an elementary substructure $M^*$ of $(V_{j(\kappa^*)})^{M_{\omega_1^N}}$ that $N$ thinks is countable and transitive, containing some $N$-ordinal $\zeta$ which is really ill-founded.

Now we apply Theorem \ref{wftrees} in $N$.  Let $T \in N$ be a sequential tree such that $N$ thinks $T$ is well-founded and of rank $\zeta$.  In $N$, there is an honest subcategory of $\mathcal A_\delta$ isomorphic to $T$.  The honesty is absolute to our universe, since the incomparability is just witnessed by inequalities among the ordertypes of the first $\omega$ measurable cardinals of the various models corresponding to the nodes of $T$.

To complete the argument, we show that in our universe, there is a subtree of $T$ isomorphic to the complete binary tree of height $\omega$.  It suffices to show that if $S \in N$ is a sequential tree such that $N \models \rho_S(\emptyset) = \zeta$, where $\zeta$ is really ill-founded, then there are incomparable nodes $p,q \in S$ such that $\rho_S(p)$ and $\rho_S(q)$ are both ill-founded ordinals, for then the conclusion follows by a simple induction.  Now since $N \models \rho_S(t) = \sup_{s < t}(\rho_S(s) + 1)$ for all $t \in S$, there is $p \in S$ of rank $i(\zeta)$.  Since $i(\zeta)$ is a limit ordinal, $p$ must have infinitely many nodes immediately below it, call them $\{ q_i : i <\omega \}$, and $N \models \rho_S(p) = \sup_i(\rho_S(q_i) + 1)$.  Thus there are $n<m$ such that $\rho_S(q_n)$ and $\rho_S(q_m)$ are both ill-founded, as there is no least ill-founded ordinal.
\end{proof}

\section{Questions}

In an earlier version of this manuscript, we asked whether there can be two models of \ZFC with the same ordinals such that each is elementarily embeddable into the other.  This was answered in \cite{2021arXiv210812355E}.

If $M$ is a self-embeddable inner model, then in light of Theorems \ref{kunengen} and \ref{agreement}, it is interesting to ask how close $M$ can be to $V$.  The second author \cite{capture} showed under some mild large cardinal assumptions that $V$ can be a class-generic forcing extension of such an $M$.  Some related questions include:
\begin{question}
Suppose $M$ is a transitive proper class definable from parameters.
\begin{enumerate}
\item Can $M$ be self-embeddable and correct about cardinals?
\item Can $M$ be self-embeddable and correct about cofinalities?
\item Can $M$ be embeddable into $V$ and be correct about either cardinals or cofinality $\omega$?
\end{enumerate}
\end{question}

Our direct-limit constructions in Section \ref{amcat} were particular to countable models.  Thus it is natural to ask:
\begin{question}
Is it consistent that there is an ordinal $\delta$ of uncountable cofinality such that $\mathcal A_\delta$ contains a copy of the linear order $\omega+1$?
\end{question}

Although we ruled out the possibility of a Suslin tree being isomorphic to a subcategory of $\mathcal A_\delta$ for a countable $\delta$, we don't know much else about the combinatorial properties of the Aronszajn tree constructed in Theorem \ref{aronszajn}. 
\begin{question}
Characterize the class of Aronszajn trees that can be isomorphic to an honest subcategory of $\mathcal A_\delta$ for a countable $\delta$.  Can it contain non-special trees?
\end{question}

Our methods for building upward-growing and downward-growing trees as honest subcategories of $\mathcal A_\delta$ are quite different, and we don't know if they can be combined in any interesting way.  We would like to know if there is any structural asymmetry in this category, or rather:

\begin{question}Can an $\mathcal A_\delta$ contain an honest subcategory isomorphic to a ``diamond'', i.e. the standard partial order on the four-element boolean algebra?
\end{question}

\begin{question}
If $\mathbb P$ is a partial order isomorphic to an honest subcategory of $\mathcal A_\delta$, does this hold for the reverse of $\mathbb P$?
\end{question}

\bibliographystyle{amsplain.bst}
\bibliography{MRbib.bib}

\end{document}